\documentclass{arxiv}

\usepackage{xcolor}
\usepackage{enumitem}
\usepackage{comment}
\usepackage[verbose]{hyperref}
\hypersetup{colorlinks=false,allbordercolors=blue,pdfborderstyle={/S/U/W 1}}

\begin{document}

%%%%%%%%%%%%%%%%%%%%% Publisher's Area please ignore %%%%%%%%%%%%%%%
%

%
%%%%%%%%%%%%%%%%%%%%%%%%%%%%%%%%%%%%%%%%%%%%%%%%%%%%%%%%%%%%%%%%%%%%

\title{On $(\delta,f)$-derivations and Jordan $(\delta,f)$-derivations on modules}

\author{\footnotesize Gusti Ayu Dwi Yanti $^{\dagger,+}$ and Indah Emilia Wijayanti $^{\dagger,}$\footnote{Corresponding author}} 

\address{$^{\dagger}$Department of Mathematics\\
Faculty of Mathematics and Natural Sciences\\
Universitas Gadjah Mada\\
Yogyakarta, Indonesia\\ 
$^+$gustiayudwiyanti2001@mail.ugm.ac.id, gustiayu9d@gmail.com\\
$^*$ind$\_$wijayanti@ugm.ac.id}

\maketitle

\begin{abstract}
Let $R$ be a ring with identity, $M,N$ right modules over $R$. An additive mapping $\delta$ from $R$ to $R$ is called derivation on ring $R$ if it satisfies the Leibniz condition. If $\delta$ is a derivation on $R$ and $f:M \rightarrow N$ is a module homomorphism over $R$, then an additive mapping $d:M \rightarrow N$ is called a $(\delta,f)$-derivation if it satisfies $d(xa)=d(x)a+f(x)\delta(a)$ for all $x \in M$ and $a \in R$. An additive mapping $\delta: R \rightarrow R$ is called Jordan derivation on ring $R$ if $\delta(x^2)=\delta(x)x+x\delta(x)$ for all $x \in R$, which is the generalization of derivation  This paper presents generalization of Posner's First Theorem of $(\delta,f)$-derivation on $2$-torsion prime modules. It also provides a generalization of some results in case of $2$-torsion free prime modules from ring situation. Moreover, we introduce a Jordan $(\delta,f)$-derivation on modules and prove that every Jordan $(\delta,f)$-derivation on modules is a $(\delta,f)$-derivation on modules.
\end{abstract}


\ccode{2020 Mathematics Subject Classification: 47B47, 13F20, 16S34}

\section{Introduction}	

Throughout this paper, the ring means an associative ring with identity, and all modules are considered to be unital. Jacobson \cite{jacobson} defined derivations on rings using Leibniz rule for the multiplication of two elements within the ring. An additive mapping $\delta$ is called a derivation on ring $R$ if $\delta (ab)=\delta (a) b + a \delta (b)$ for all $a,b \in R$. In addition, several types of derivation are presented in Ali et al \cite{shakir}.

Some results from Posner's studies in \cite{posner} have motivated some previous authors to make the generalization. Posner proved that in a prime ring with characteristic not equal to two, if the composition of two derivations results in another derivation, then at least one of the original derivations must be zero. Esslamzadeh and Ghahramani \cite{esslamzadeh} investigated generalized derivations and a special case of Posner's first theorem. They showed that if the composition $\tau$ is a generalized derivations relative to derivations $\delta$ and $\phi$ is an endomorphism of a prime module $M$ over an algebra $\mathcal{A}$, such that $\tau \circ \phi$ is an endomorphism of a prime module $M$ over $\mathcal{A}$, then either $\tau$ is an endomorphism of the prime module $M$ over $\mathcal{A}$ or $\phi$ must be zero.

Ghahramani et al. \cite{ghahramani} studied a generalization of Posner's first theorem to generalized derivations on $2$-torsion free prime modules and proved that the composition of two generalized derivations relative to the composition of two additive maps if and only if one of them is zero, of if both are nonzero, then they must be endomorphisms of the module $M$. Moreover, motivated by Posner's first Theorem, Creedon \cite{creedon} showed that if $\delta_1$ and $\delta_2$ are derivations on a ring $R$ such that $\delta_1 \delta_2 (R) \subseteq P$, where $P$ is a prime ideal and the characteristic of the quotient ring $R/P$ is not $2$, then it must follows that either $\delta_1(R) \subseteq P$ or $\delta_2 (R) \subseteq P$.

Furthermore, Bland \cite{bland} has defined the concept of $(\delta,f)$-derivation on a module as following:
\begin{definition}\label{def1.1.1} 
    Let $R$ be a ring with identity, and $M,N$ be right modules over ring $R$. Suppose $\delta$ is a derivation on $R$, and $f: M \rightarrow N$ is a module homomorphism over $R$. Then an additive map $d: M \rightarrow N$ is called a $(\delta,f)$-derivation if it satisfies $d(xa)=d(x)a+f(x)\delta(a)$ for all $x \in M$ and $a \in R$.
\end{definition}

We refer to $(\delta,f)$-derivation simply as $f$-derivation, with the understanding that $\delta$ is fixed. In a special case, i.e.  if $f : M \rightarrow M$, then we call the $(\delta,f)$-derivation $d: M \rightarrow M$ as $f$-derivation on $M$. In particular, if $f=id_M$ then $d$ becomes a derivation on $M$. If $f=0$, then $d$ reduces to a module homomorphism over $R$. Bland's research on $(\delta,f)$-derivation focuses on the localization and colocalization functor in the category of unitary right module $M$ over ring $R$. Following the notion of $(\delta,f)$-derivation by Bland, Fitriani et al \cite{fitriani} observed $(\delta,f)$-derivation on polynomial modules $M[x]$ and applied the results to semigroup ring $R[S]$ and semigroup module $M[S]$. 

On the other hand, Herstein \cite{heirstein} observed the relationship between associative rings and the Jordan product and defined a Jordan derivation. An additive mapping is called Jordan derivation on $R$ if $\delta (a^2) = \delta (a)a+a \delta (a)$ for all $a \in R$. Jordan derivation is the generalization of derivation because if $\delta$ is a derivations, then $\delta$ is also a Jordan derivation. But the converse is not always the case. Some authors study further Jordan derivation, for example Cusack \cite{cusack} who observed Jordan derivation on rings with characteristic not equal to two and non zero divisor commutator, then every Jordan derivation is a derivation. Moreover, Li \cite{li} defined Jordan generalized derivation from associative algebra $\mathcal{A}$ over commutative ring $R$ to bimodule $\mathcal{M}$ over $\mathcal{A}$ relative to a linear map from algebra $\mathcal{A}$ over $R$ to bimodule $\mathcal{M}$ over $\mathcal{A}$.

In this paper, we investigate the generalization of Posner's first Theorem of $(\delta,f)$-derivation on $2$-torsion prime modules. In addition, we provide a generalization of Creedon's Theorem of $(\delta,f)$-derivation on the $2$-torsion free prime modules. We induce Jordan $(\delta,f)$-derivation on module is a generalization of a $(\delta,f)$-derivation. Finally, we prove that every Jordan $(\delta,f)$-derivation on modules is an $(\delta,f)$-derivation on modules.

\section{On $(\delta,f)$-derivations on prime modules}

In this section, we study the relationship between the primeness of rings and modules and the generalization of Posner's theorem using some Creedon's result of $(\delta,f)$-derivation \cite{creedon}. 

A ring $R$ is said to be prime if for all $x,y \in R$ if $xRy=0$, then it implies that either $x=0$ or $y=0$. Assume that $M$ is a right module over ring $R$. The set 
\begin{align*}
    Ann_R^r(M)=\{r \in R \mid \text{ for all } m \in M, mr=0_R\}
\end{align*} 
is called the right annihilator of $M$ in ring $R$. For any $K$ a submodule of $M$, We define the quotient 
\begin{align*}
    (K:M)=\{ r \in R \mid \text{ for all } m \in M, mr \in K\} = Ann_R^r(M/K)
\end{align*}
which is an ideal of ring $R$. In case $K=0$, then $(0:M)=Ann_R^r(M)$. A proper submodule $K$ of $M$ is called a prime submodule, if for all $x \in R$ and $m \in M, mRx \subseteq K$ implies that either $m \in K$ or $x \in (K:M)$. A module $M$ is called a prime module if the zero submodule $0$ is a prime submodule. We refer to \cite{indah} for more explanation of the primeness of submodules.

Module $M$ is said to be faithful if $Ann_R^r (M) = \{ r \in R \mid \text{ for all } m \in M, mr = 0_R \} = 0_R$. Let $\overline{R}=R/Ann_R^r(M)$. We recall \cite[Proposition 2.2 (i)]{indah}, that is if $M$ is prime, then $\overline{R}$ is a prime ring. So, if $Ann_R^r(M)=0$ then $R$ is a prime ring. Moreover, $M$ can be regarded as a module over the ring $\overline{R}$. Module $M$ is called $2$-torsion free if for all $m \in M$ if $2m=0$ implies that $m=0$. If $M$ is a $2$-torsion free, then the factor ring $\overline{R}$ is also $2$-torsion free.

Now we consider the following example. 
\begin{example}\label{example2.1}
    Given the polynomial ring $\mathbb{Z}[x]$ and module  $\mathbb{Z}[x] \times \mathbb{Z}[x]$ over $\mathbb{Z}[x]$. We define derivations $\delta_1,\delta_2:\mathbb{Z}[x] \rightarrow \mathbb{Z}[x]$ by $\delta_1(a(x))=\delta(a_0+a_1x+a_2x^2+\cdots +a_nx^n)=a_1+2a_2x+\cdots+na_nx^{n-1}=a'(x)$ and $\delta_2(a(x))=qa'(x)$ for every $a(x)=a_0+a_1x+a_2x^2+\cdots+a_nx^n \in \mathbb{Z}[x]$, where $q$ is a fixed element in $ \in \mathbb{Z}$. Moreover, we define module homomorphisms $f_1,f_2:\mathbb{Z}[x] \times \mathbb{Z}[x] \rightarrow \mathbb{Z}[x] \times \mathbb{Z}[x]$ by $f_1 \left( \begin{bmatrix}
        a(x)\\
        b(x)
    \end{bmatrix} \right) = \begin{bmatrix}
        a(x)\\
        b(x)
    \end{bmatrix}$ and $f_2 \left( \begin{bmatrix}
        a(x)\\
        b(x)
    \end{bmatrix} \right) =\begin{bmatrix}
        \frac{p}{q} a(x)\\
        \frac{1}{q} b(x)
    \end{bmatrix}$ for every $\begin{bmatrix}
        a(x)\\
        b(x)
    \end{bmatrix} \in \mathbb{Z}[x] \times \mathbb{Z}[x]$, where $p,q$ are fixed elements in $ \mathbb{Z}$. Based on these conditions we define $(\delta_i,f_i)$-derivation for $i=1,2$ as following : $d_i : \mathbb{Z}[x] \times \mathbb{Z}[x] \rightarrow \mathbb{Z}[x] \times \mathbb{Z}[x]$, where $d_1 \left( \begin{bmatrix}
        a(x)\\
        b(x)
    \end{bmatrix} \right) = \begin{bmatrix}
        a'(x)\\
        b'(x)
    \end{bmatrix}$ and $d_2 \left( \begin{bmatrix}
        a(x)\\
        b(x)
    \end{bmatrix} \right) = \begin{bmatrix}
        pa'(x) + a(x)\\
        b(x)
    \end{bmatrix}$ for every $\begin{bmatrix}
        a(x)\\
        b(x)
    \end{bmatrix} \in \mathbb{Z}[x] \times \mathbb{Z}[x]$ for some fixed element $p \in \mathbb{Z}$. 
    
     Let $\begin{bmatrix}
        a(x)\\
        b(x)
    \end{bmatrix} \in \mathbb{Z}[x] \times \mathbb{Z}[x]$ and $c(x) \in \mathbb{Z}[x]$. We obtain
    \begin{align*}
        d_1d_2 \left( \begin{bmatrix}
        a(x)\\
        b(x)
    \end{bmatrix} c(x) \right) &= d_1d_2 \left( \begin{bmatrix}
        a(x)c(x)\\
        b(x)c(x)
    \end{bmatrix} \right) \\
    &= d_1 \left( \begin{bmatrix}
        p(a'(x)c(x)+a(x)c'(x))+a(x)c(x)\\
        b(x)c(x)
    \end{bmatrix} \right)\\
    &= \begin{bmatrix}
        p(a''(x)c(x)+2a'(x)c'(x)+a(x)c''(x))+a'(x)c(x)+a(x)c'(x)\\
        b'(x)c(x)+b(x)c'(x)
    \end{bmatrix}
    \end{align*}
    and
    \begin{align*}
        &d_1d_2 \left(\begin{bmatrix}
        a(x)\\
        b(x)
    \end{bmatrix}\right) c(x) + f_1f_2 \left(\begin{bmatrix}
        a(x)\\
        b(x)
    \end{bmatrix} \right) \delta_1\delta_2(c(x))\\
    &= d_1 \left( \begin{bmatrix}
        pa'(x)+a(x)\\
        b(x)
    \end{bmatrix} \right)c(x) +f_1 \left( \begin{bmatrix}
        \frac{p}{q} a(x)\\
        \frac{1}{q} b(x)
    \end{bmatrix} \right) \delta_1(qc'(x))\\
    &= \begin{bmatrix}
        pa''(x)+a'(x)\\
        b'(x)
    \end{bmatrix} c(x) + \begin{bmatrix}
        \frac{p}{q} a(x)\\
        \frac{1}{q} b(x)
    \end{bmatrix} qc''(x)\\
    &= \begin{bmatrix}
        pa''(x)c(x)+a'(x)c(x)+pa(x)c'(x)\\
        b'(x)c(x) + bc''(x)
    \end{bmatrix}.
    \end{align*}
    Because $d_1d_2 \left( \begin{bmatrix}
        a(x)\\
        b(x)
    \end{bmatrix} c(x) \right) \ne d_1d_2 \left(\begin{bmatrix}
        a(x)\\
        b(x)
    \end{bmatrix}\right) c(x) + f_1f_2 \left(\begin{bmatrix}
        a(x)\\
        b(x)
    \end{bmatrix} \right) \delta_1\delta_2(c(x))$, we conclude that $d_1d_2$ is not an $(\delta_1\delta_2,f_1f_2)$-derivation.
\end{example}
In Example \ref{example2.1} we show that there are non-zero derivations $d_1$ and $ d_2$ on $\mathbb{Z}[x] \times \mathbb{Z}[x]$ in which the composition $d_1d_2$ is not a $(\delta_1\delta_2,f_1f_2)$-derivation.

Based on the fact in the Example \ref{example2.1}, we provide the necessary and sufficient conditions for $d_1$ and $d_2$ to obtain a $(\delta_1\delta_2,f_1f_2)$-derivation.
\begin{theorem}\label{teo3.1.1}
    Let $R$ be a ring, $M$ a prime right $R$-module and $2$-torsion free. Let $\delta_i : R \rightarrow R$ be a derivation on $R$, $f_i: M \rightarrow M$ an $R$-module epimorphism and $d_i:M \rightarrow M$ a $(\delta_i,f_i)$-derivation on $M$, $i=1,2$. Then $d_1d_2$ is a $(\delta_1\delta_2,f_1f_2)$-derivation on M  if and only if one of the following conditions holds:
    \begin{enumerate}
        \item $d_1 = 0$,
        \item $d_2 = 0$,
        \item $d_1 \ne 0, d_2 \ne 0$ and $d_1,d_2 \in End_R(M)$.
    \end{enumerate}
\end{theorem}
\begin{proof}
    Based on \cite[Proposition 2.2 (i)]{indah}, we have $\overline{R}$ is a prime ring. Moreover, $M$ is a right module over $\overline{R}$.

    Let $x \in Ann_R^r(M)$ and $m \in M$. We obtain
    \begin{align*}
        d_i (mx)=d_i(m)x+f_i(m)\delta_i(x), ~~ i=1,2.
    \end{align*}
    So $f_i(m)\delta_i(x)=0$. Because $f_2$ is an $R$-module epimorphism, we have $\delta_i(Ann_R^r(M)) \subseteq Ann_R^r(M)$. Thus, we can define a map $\delta_i^* : \overline{R} \rightarrow \overline{R}$ for $i=1,2$ by $\delta_i^* (x+Ann_R^r(M))=\delta_i(x)+Ann_R^r(M)$ for all $x \in R$. Because $\delta_i(Ann_R^r(M)) \subseteq Ann_R^r(M)$, this definition is well-defined. Furthermore, each $\delta_i^*$ is an additive mapping. Now, let $m \in M, x \in R$ and for $i=1,2$, we have
    \begin{align*}
        d_i(m(x+Ann_R^r(M))) &= d_i(mx)\\
        &= d_i(m)x+f_i(m)\delta_i (x) \\
        &= d_i(m)(x+Ann_R^r(M)) + f_i(m)(\delta_i(x)+Ann_R^r(M))\\
        &= d_i(m)(x+Ann_R^r(M)) + f_i(m)\delta_i^*(x+Ann_R^r(M)).
    \end{align*}
    So each $d_i$ is $(\delta_i^*,f_i)$-derivation and since $Ann_{\overline{R}}^r (M)=(0)$, it follows that each $\delta_i^*$ is derivation on $\overline{R}$. Additionaly, it is straighforward to verify that $(\delta_1\delta_2)^* = \delta_1^* \delta_2^*$. So,  $\delta_1^* \delta_2^*$ is a derivation on $\overline{R}$. Given that $d_1d_2$ is $(\delta_1\delta_2,f_1f_2)$-derivation and using a similar argument as before, we conclude that $d_1d_2$ is also $( \delta_1^* \delta_2^*,f_1f_2)$-derivation. Moreover, by Posner's first theorem, either $\delta_1^*=0$ or $\delta_2^*=0$.

    Assumed $\delta_1^*=0$. By the definition of $\delta_1^*$, we have $\delta_1(R)\subseteq Ann_R^r(M)$. Consequently $d_1(mx)=d_1(m)x$ for all $m \in M$ and $x \in R$, which shows that $d_1 \in End_R M$. Therefore, 
    \begin{align*}
        d_1d_2(mx)&= d_1(d_2(m)x+f_2(m)\delta_2(x))\\
        &= d_1d_2(m)x+d_1(f_2(m))\delta_2(x).
    \end{align*}
    Meanwhile, using the assumption and the fact that $\delta_1(R)\subseteq Ann_R^r(M)$, we also have
    \begin{align*}
        d_1d_2(mx)&=d_1d_2(m)x+f_1f_2(m)\delta_1\delta_2(x)\\
        &= d_1d_2(m)x.
    \end{align*}
    Comparing the two expressions for $d_1d_2(mx)$, we conclude
    \begin{align*}
        d_1(f_2(m))\delta_2(x) = 0.
    \end{align*}
    Thus, for all $m \in M$ and $x,y\in R$, we get
    \begin{align*}
        d_1(f_2(m))x\delta_2(x)=d_1(f_2(m)x)\delta_2(x)=0,
    \end{align*}
    which implies that $d_1(f_2(m))R\delta_2(x)=0$ for all $m \in M$ and $x \in R$. Since $M$ is prime module, this yields either $d_1(f_2(m))=0$ or $M \delta_2(x)=0$ for all $m \in M$ and $x \in R$. Therefore, we conclude that either $d_1=0$, or $d_1 \ne 0$ and $M\delta_2(x) =0$ for all $x \in R$, which again means $d_1=0$ or $d_1 \ne 0$ and $d_2 \in End_R(M)$.

    By similar argument as above, if $\delta_2^*=0$, we obtain that either $d_2=0$ or $d_2 \ne 0$ and $d_1 \in End_R(M)$.

    As a result, one of ther alternatives $(1),(2),(3)$ must be true.
    \\
    Conversely, suppose that case $(2)$ holds. So $f_2(m)\delta_2(x)=0$ and we have $\delta_2(R)\subseteq Ann_R^r(M)$. Using a similiar argument as before, we can also establish that $\delta_1(Ann_R^r(M))\subseteq Ann_R^r(M)$. Therefore $d_1d_2=0$ and we obtain
    \begin{align*}
        0 = d_1d_2(mx)=d_1d_2(m)x+f_1f_2(m)\delta_1\delta_2(x)
    \end{align*}
    for all $m \in M$ and $x \in R$. This means that $d_1d_2$ is a $(\delta_1\delta_2,f_1f_2)$-derivation.

    If either case $(1)$ or $(3)$ occur instead, the same line of argument proves that $d_1d_2$ is also a  $(\delta_1\delta_2,f_1f_2)$-derivation.
\end{proof}

The following example shows us the fact that the composition of a non-zero derivation and a non-zero  endomorphism is not necessary an endomorphism. 
\begin{example}\label{example2.2}
We recall the notions in Example  \ref{example2.1} and 
    consider  $(\delta_1,f_1)$-derivation $d:\mathbb{Z}[x] \times \mathbb{Z}[x] \rightarrow \mathbb{Z}[x] \times \mathbb{Z}[x]$ as in the Example \ref{example2.1}. We define an endomorphism $\gamma: \mathbb{Z}[x] \times \mathbb{Z}[x] \rightarrow \mathbb{Z}[x] \times \mathbb{Z}[x]$ over $\mathbb{Z}[x]$ as  $\gamma \left( \begin{bmatrix}
        a(x)\\
        b(x)
    \end{bmatrix} \right) = \begin{bmatrix}
        2a(x)+3b(x)\\
        a(x)
    \end{bmatrix}$.

    Let $\begin{bmatrix}
        a(x)\\
        b(x)
    \end{bmatrix} \in \mathbb{Z}[x] \times \mathbb{Z}[x]$ and $c(x) \in \mathbb{Z}[x]$. We obtain
    \begin{align*}
        d_1\gamma \left( \begin{bmatrix}
            a(x)\\
            b(x)
        \end{bmatrix} c(x) \right) &= d_1 \gamma \left( \begin{bmatrix}
            a(x)c(x)\\
            b(x)c(x)
        \end{bmatrix} \right) \\
        &= d_1 \left( \begin{bmatrix}
        2a(x)c(x)+3b(x)c(x)\\
        a(x)c(x)
        \end{bmatrix} \right) \\
        &= \begin{bmatrix}
            2(a'(x)c(x)+a(x)c'(x))+3(b'(x)c(x)+b(x)c'(x))\\
            a'(x)c(x)+a(x)c'(x)
        \end{bmatrix}
    \end{align*}
    and
    \begin{align*}
        d_1 \gamma \left( \begin{bmatrix}
            a(x)\\
            b(x)
        \end{bmatrix} \right) c(x) &= d_1 \left( \begin{bmatrix}
            2a(x)+3b(x)\\
            a(x)
        \end{bmatrix} \right)c(x)\\
        &= \begin{bmatrix}
            2a'(x)+3b'(x)\\
            a'(x)
        \end{bmatrix} c(x) = \begin{bmatrix}
            2a'(x)c(x)+3b'(x)c(x)\\
            a'(x)c(x)
        \end{bmatrix}.
    \end{align*}
    Since $ d_1\gamma \left( \begin{bmatrix}
            a(x)\\
            b(x)
        \end{bmatrix} c(x) \right) \ne  d_1 \gamma \left( \begin{bmatrix}
            a(x)\\
            b(x)
        \end{bmatrix} \right) c(x) $, we conclude that $d_1 \gamma \notin End_R(M)$.
\end{example}

As corollary of Theorem \ref{teo3.1.1} we obtain the following necessary and sufficient condition.

\begin{corollary}\label{remark3.1.1}
    Let $R$ be a ring, $M$ a right prime module and $2$-torsion free. Let $\delta$ be a derivation on $R$,  $f: M \rightarrow M$ an $R$-module epimorphism, $\gamma$ an endomorphism on $M$ and a $(\delta,f)$-derivation $d:M \rightarrow M$ on $M$. Then $d\gamma \in End_R (M)$ if and only if one of the following conditions hold:
    \begin{enumerate}
        \item $d=0$,
        \item $\gamma = 0$,
        \item $d\ne 0, \gamma \ne 0$ and $d \in End_R(M)$.
    \end{enumerate}
\end{corollary}
Another corollary  of Theorem \ref{teo3.1.1} is the special case in $R$ as a module over itself. 
\begin{theorem}\label{thPos1}
    Let $R$ be a prime ring $2$-torsion free,  $\delta_i : R \rightarrow R$ a derivation on $R$, $f_i: R \rightarrow R$ an $R$-module epimorphism and  $d_i:R \rightarrow R$ a  $(\delta_i,f_i)$-derivation
    on $R$ for $i=1,2$. If $d_1d_2$ is $(\delta_1\delta_2,f_1f_2)$-derivation then $d_1 = 0$ or $d_2=0$.
\end{theorem} 
    Theorem \ref{thPos1} gives the generalization of the first Posner's Theorem in case of $(\delta,f)$-derivations on rings.

In the following lemma, we show some further properties of prime submodules. 
\begin{lemma}\label{idealprima}
    Let $R$ be a ring and $M$ a right module over $R$. If $L$ is a prime submodule of $M$, then $(L:M)$ is a prime ideal of  $R$.
\end{lemma}
\begin{proof}
    Suppose $L$ is a prime submodule of $M$. First, we show that $(L:M)$ is an ideal. By definition, for any $r,s \in (L:M)$ and $m \in M$. Since $mr, ms \in L$,  we have $m(r-s)=mr-ms \in L$. For any $r \in (L:M), a\in R$ and $m \in M$ then $m(ar)=(ma)r \in L$ and $m(ra)=(mr)a \in L$. Therefore $ar,ra \in (L:M)$. So, $(L:M)$ is an ideal of $R$.\\
    Next, show that $(L:M)$ is prime ideal of $R$. Let $a,b \in R$ satisfy $aRb \subseteq (L:M)$. This means that for all $m \in M$ and $r \in R, m(arb)=(ma)rb \in L$. Since $L$ is prime submodule, this implies that $ma \in L$ or $b \in Ann_R^r(M/L)$. If $ma \in L$ for every $m \in M$, then by the definition of $(L:M)$, it follows that $a \in (L:M)$. Otherwise, if $b \in Ann_R^r(M/L)$ then $b \in (L:M)$. Therefore, $(L:M)$ is a prime ideal of $R$.
\end{proof}

Creedon \cite[Theorem 2]{creedon} showed that if $\delta_1$ and $\delta_2$ are derivations on a ring $R$ such that their composition $\delta_1 \delta_2 (R)$ is contained within a prime ideal $P$, and if the ring $R/P$ is $2$-torsion free, then at least one of the derivations satisfies $\delta_1(R) \subseteq P$ or $\delta_2 (R) \subseteq P$. 

We extend Creedon's result to right modules as we can see in the following theorem.
\begin{theorem}\label{teo3.1.2}
    Let $R$ be a ring, $M$ a right module over $R$ and $L$ a prime submodule of $M$ such that $M/L$ is a $2$-torsion free module over $\tilde{R}=R/(L:M)$. Let $\delta_i : R \rightarrow R$ be a derivation on $R$, $f_i: M \rightarrow M$ an $R$-module epimorphism and $d_i:M \rightarrow M$ a $(\delta_i,f_i)$-derivation on $M$, for $i=1,2$. If $d_1d_2(M) \subseteq L$ and $\delta_1\delta_2(R)\subseteq (L:M)$, then one of the following conditions holds:
    \begin{enumerate}
        \item $d_1(M)\subseteq L$;
        \item $d_2(M) \subseteq L$;
        \item $d_1(M), d_2(M) \nsubseteq L$ and $\delta_1(R), \delta_2(R)\subseteq (L:M)$.
    \end{enumerate}
\end{theorem}
\begin{proof}
    Assume that $L$ is a prime submodule of $M$ such that $M/L$ is $2$-torsion free over $\tilde{R}$. Then $\tilde{R}$ is a $2$-torsion free ring. Furthermore, by Lemma \ref{idealprima} we obtain $(L:M)$ is a prime ideal. Moreover, we assume that $\delta_1\delta_2(R)\subseteq (L:M)$. Based on Creedon's Theorem, it follows that $\delta_1(R) \subseteq (L:M)$ or $\delta_2(R) \subseteq (L:M)$.

    Suppose $\delta_1(R)\subseteq (L:M)$. For any $m \in M$ and $x \in R$ we obtain
    \begin{align*}
        d_1d_2(mx)&= d_1(d_2(m)x+f_2(m)\delta_2(x))\\
        &= d_1d_2(m)x+f_1(d_2(m))\delta_1(x)+d_1(f_2(m))\delta_2(x)+f_1f_2(m)\delta_1\delta_2(x).
    \end{align*}
    Based on the hypothesis, we conclude that $d_1(f_2(m))\delta_2(x) \in L$. Now we consider 
    \begin{align*}
        d_1(f_2(m)x)\delta_2(x) = d_1(f_2(m))x\delta_2(x) + f_1f_2(m)\delta_1(x)\delta_2(x).
    \end{align*}
    Since $f_1f_2(m)\delta_1(x) \in L$ and $d_1(f_2(m)x)\delta_2(x) \in L$, we have $d_1(f_2(m))x\delta_2(x) \in L$ for all $f_2(m)\in M$ and $x, \delta_2(x) \in R$. Thus, 
    \begin{align*}
        d_1(f_2(m))R\delta_2(x) \subseteq L
    \end{align*}
    for all $f_2(m)\in M$ and $\delta_2(x) \in R$. Since $L$ is a prime submodule, there are two possible cases: either $d_1(f_2(m)) \in L$ or  $M\delta_2(x) \subseteq L$ for all $x \in R$ and $m \in M$. Since $f_2$ is an $R$-module epimorphism, we have $d_1(M)\subseteq L$ or  $M\delta_2(x) \subseteq L$. In other words, we must have either $d_1(M)\subseteq L$ or  $d_1(M)\nsubseteq L$ and $\delta_2(R) \subseteq (L:M)$.

    Similarly, if $\delta_2(R)\subseteq (L:M)$, then either $d_2(M)\subseteq L$ or $d_2(M)\subseteq L$ and $\delta_1(R) \subseteq (L:M)$.
\end{proof}

The conditions of $\delta_1(R) \nsubseteq (L:M), \delta_2(R) \nsubseteq (L:M)$ 
in Theorem \ref{teo3.1.2} are necessary. We give an example, in which $d_1(M) , d_2 (M) \subseteq L$, 
$\delta_1(R) \nsubseteq (L:M)$, 
$ \delta_2(R) \nsubseteq (L:M)$, but   $d_1d_2(M) \subseteq L$ and $\delta_1 \delta_2 (R) \nsubseteq (L:M)$.
\begin{example}\label{example2.3}
    We consider the polynomial ring $\mathbb{Z}[x]$, module  $\mathbb{Z}[x] \times \mathbb{Z}[x]$ over $\mathbb{Z}[x]$ and $L = \left\{ \begin{bmatrix}
        a(x)\\
        0
    \end{bmatrix} \mid a(x) \in \mathbb{Z}[x] \right\}$ is a prime submodule of $\mathbb{Z}[x] \times \mathbb{Z}[x]$. We define derivation $\delta_1,\delta_2:\mathbb{Z}[x] \rightarrow \mathbb{Z}[x]$ by $\delta_1(a(x))=\delta(a_0+a_1x+a_2x^2+\cdots +a_nx^n)=a_1+2a_2x+\cdots+na_nx^{n-1}=a'(x)$ and $\delta_2(a(x))=qa'(x)$ for every $a(x)=a_0+a_1x+a_2x^2+\cdots+a_nx^n \in \mathbb{Z}[x]$ and fixed element $q \in \mathbb{Z}$. 
    We define  module homomorphism $f_1,f_2:\mathbb{Z}[x] \times \mathbb{Z}[x] \rightarrow \mathbb{Z}[x] \times \mathbb{Z}[x]$ over $\mathbb{Z}[x]$ by $f_1 \left( \begin{bmatrix}
        a(x)\\
        b(x)
    \end{bmatrix} \right) = \begin{bmatrix}
        a(x)\\
        0
    \end{bmatrix}$ and $f_2 \left( \begin{bmatrix}
        a(x)\\
        b(x)
    \end{bmatrix} \right) =\begin{bmatrix}
        \frac{p}{q} a(x)\\
         0
    \end{bmatrix}$ for every $\begin{bmatrix}
        a(x)\\
        b(x)
    \end{bmatrix} \in \mathbb{Z}[x] \times \mathbb{Z}[x]$ and fixed element $p,q \in \mathbb{Z}$.

    Moreover, $d_i : \mathbb{Z}[x] \times \mathbb{Z}[x] \rightarrow \mathbb{Z}[x] \times \mathbb{Z}[x]$  is a $(\delta_i,f_i)$-derivation, $i=1,2$, where $d_1 \left( \begin{bmatrix}
        a(x)\\
        b(x)
    \end{bmatrix} \right) = \begin{bmatrix}
        a'(x)\\
       0
    \end{bmatrix}$ and $d_2 \left( \begin{bmatrix}
        a(x)\\
        b(x)
    \end{bmatrix} \right) = \begin{bmatrix}
        pa'(x) + a(x)\\
        0
    \end{bmatrix}$ for every $\begin{bmatrix}
        a(x)\\
        b(x)
    \end{bmatrix} \in \mathbb{Z}[x] \times \mathbb{Z}[x]$ and fixed element $p \in \mathbb{Z}$. We will show that if $d_1(M) \subseteq L, d_2 (M) \subseteq L$ and $\delta_1(R) \nsubseteq (L:M), \delta_2(R) \nsubseteq (L:M)$ then $d_1d_2(M) \subseteq L$ and $\delta_1 \delta_2 (R) \nsubseteq (L:M)$.\\
    Let $\begin{bmatrix}
        a(x)\\
        b(x)
    \end{bmatrix} \in \mathbb{Z}[x] \times \mathbb{Z}[x]$ and $c(x) \in \mathbb{Z}[x]$, we obtain
    \begin{align*}
        d_1d_2 \left( \begin{bmatrix}
        a(x)\\
        b(x)
    \end{bmatrix} \right)
    &= d_1 \left( d_2 \left( \begin{bmatrix}
        a(x)\\
        b(x)
    \end{bmatrix} \right) \right) 
    = d_1 \left( \begin{bmatrix}
        pa'(x)+a(x)\\
        0
    \end{bmatrix} \right)
    = \begin{bmatrix}
        pa''(x)a'(x)\\
        0
    \end{bmatrix} \in L
    \end{align*}
    and
    \begin{align*}
        \delta_1\delta_2(c(x))= \delta_1(\delta_2(c(x)))= \delta_1(qc'(x))=qc''(x) \nsubseteq (L:M).
    \end{align*}
    Therefore, if $d_1(M) \subseteq L, d_2 (M) \subseteq L$ and $\delta_1(R) \nsubseteq (L:M), \delta_2(R) \nsubseteq (L:M)$, but $d_1d_2(M) \subseteq L$ and $\delta_1 \delta_2 (R) \nsubseteq (L:M)$.
\end{example}

We give some conditions of the composition of a $(\delta,f)$-derivation $d$ on $M$ and a module endomorphism $\gamma$ to be a homomorphism.

\begin{lemma} \label{remark3.1.2}
   Let $R$ be a ring, $M$ be a right module over ring $R$ and $L$ a prime submodule of $M$ such that $M/P$ is $2$-torsion free over $\tilde{R}$. Let  $\delta$ be a derivation on $R$, $f$ an epimorphism on $M$  and  $d:M \rightarrow M$ a $(\delta,f)$-derivation on $M$. Given $\gamma \in End_R(M)$ such that satisfies $d\gamma=0$. Then one of the following conditions holds:
    \begin{enumerate}
        \item $d=0$,
        \item $\gamma = 0$,
        \item $d\ne 0, \gamma \ne 0$ and $d \in End_R(M)$.
    \end{enumerate} 
\end{lemma}
    
The following example shows that the conditions in Lemma \ref{remark3.1.2} are necessary.

\begin{example}\label{example2.4}
    Consider the polynomial ring $\mathbb{Z}[x]$, the right polynomial module $\mathbb{Z}[x] \times \mathbb{Z}[x]$ over $\mathbb{Z}[x]$ and $L = \left\{ \begin{bmatrix}
        a(x)\\
        0
    \end{bmatrix} \mid a(x) \in \mathbb{Z}[x] \right\}$ is prime submodule of $\mathbb{Z}[x] \times \mathbb{Z}[x]$. We define derivation $\delta_1:\mathbb{Z}[x] \rightarrow \mathbb{Z}[x]$ on $\mathbb{Z}[x]$, a module homomorphism $f_1: \mathbb{Z}[x]^2 \rightarrow \mathbb{Z}[x]^2$ over $\mathbb{Z}[x]$ and $(\delta_1,f_1)$-derivation $d_1:\mathbb{Z}[x] \times \mathbb{Z}[x] \rightarrow \mathbb{Z}[x] \times \mathbb{Z}[x]$ as in the Example \ref{example2.3}. We define an  endomorphism $\gamma: \mathbb{Z}[x] \times \mathbb{Z}[x] \rightarrow \mathbb{Z}[x] \times \mathbb{Z}[x]$ over $\mathbb{Z}[x]$ by $\gamma \left( \begin{bmatrix}
        a(x)\\
        b(x)
    \end{bmatrix} \right) = \begin{bmatrix}
        2a(x)+3b(x)\\
        0
    \end{bmatrix}$ for every $\begin{bmatrix}
        a(x)\\
        b(x)
    \end{bmatrix} \in \mathbb{Z}[x]^2$. We will show that if $d_1 \ne 0, \gamma \ne 0$ and $d \notin End_R(M)$ then $d\gamma \ne 0$.\\
    Let $\begin{bmatrix}
        a(x)\\
        b(x)
    \end{bmatrix} \in \mathbb{Z}[x] \times \mathbb{Z}[x]$, we obtain
    \begin{align*}
        d_1\gamma \left(\begin{bmatrix}
        a(x)\\
        b(x)
    \end{bmatrix} \right) = d_1 \left( \gamma \left( \begin{bmatrix}
        a(x)\\
        b(x)
    \end{bmatrix} \right) \right) = d_1 \left( \begin{bmatrix}
        2a(x)+3b(x)\\
        0
    \end{bmatrix} \right) = \begin{bmatrix}
        2a'(x) +3b'(x)\\
        0
    \end{bmatrix}.
    \end{align*} 
    Therefore, if $d_1 \ne 0, \gamma \ne 0$ and $d \notin End_R(M)$ then $d\gamma \ne 0$.
\end{example}

Following is a special case of Theorem \ref{teo3.1.2}, which is one of the main result in \cite{creedon}. 
\begin{corollary}\label{corprime}
    Let $R$ be a ring, $P$ a prime ideal of ring $R$ such that $R/P$ is a $2$-torsion free ring. 
    Let $\delta_i$ a derivation on $R$,  $f_i : R \rightarrow R$ an epimorphism and $d_i$ a  $(\delta_i,f_i)$-derivation on   $R$, $i=1,2$. If $d_1d_2(R) \subseteq P$ and $\delta_1 \delta_2 (R) \subseteq P$ then $d_1(R)\subseteq P$ or $d_2(R) \subseteq P$.
\end{corollary}

\section{Jordan $(\delta,f)$-derivation on modules}

In this section, we assume that $R$ is a commutative ring with identity and $\mathcal{S}$ is an associative algebra over $R$ with identity. We define a new operation called the Jordan product on $\mathcal{S}$, denoted by $\bullet : \mathcal{S} \times S \rightarrow \mathcal{S}$, where $a \bullet b = ab+ba$ for all $a,b \in \mathcal{S}$. Hence, $(\mathcal{S}, +, \bullet)$ is a Jordan algebra over $R$. Let $\mathcal{M}$ be a bimodule over algebra $\mathcal{S}$. 
Considering $\mathcal{S}$ as a Jordan algebra, we define a Jordan action  $\bullet':\mathcal{M} \times \mathcal{S} \rightarrow \mathcal{M}$ for all $m \in \mathcal{M}, s \in \mathcal{S}$, where $m \bullet' s = s\bullet' m = ms+sm$. Therefore, $\mathcal{M}$ is a Jordan bimodule over $\mathcal{S}$.

Using these assumptions, we present the definition of derivation on module $\mathcal{M}$ over $\mathcal{S}$ as following.
\begin{definition}\label{def3.3.(1)} Let $\delta$ be a derivation on $\mathcal{S}$ and $f: \mathcal{S} \rightarrow \mathcal{M}$ a bimodule homomorphism over $\mathcal{S}$. 
    \begin{enumerate}[label=(\roman*)]
    \item An additive mapping $D: \mathcal{S} \rightarrow \mathcal{M}$ is called a $(\delta,f)$-derivation on $\mathcal{M}$ if $D(xy)= D(x)y+f(x)\delta(y)$ for all $x,y \in \mathcal{S}$.
\item An additive mapping $D: \mathcal{S} \rightarrow \mathcal{M}$ is called a Jordan $(\delta,f)$-derivation on $\mathcal{M}$ if $D(x^2)= D(x)x+f(x)\delta(x)$ for all $x \in \mathcal{S}$.
\end{enumerate}
\end{definition}
In case the $\delta$ is clear, we call it as $f$-derivation and Jordan $f$-derivation, respectively.

An example of Jordan $(\delta,f)$-derivation on module is presented here.
\begin{example}\label{ex3.3.1}
    We consider $\mathcal{S}=M_2 (\mathbb{Q})$ as an associative algebra over $\mathbb{Q}$ with an action is defined by the multiplicative $Aq$ for all $A \in \mathcal{S} $ and $q \in \mathbb{Q}$. Moreover, $\mathcal{M}=M_2 (\mathbb{Q})$ is a bimodule over $\mathcal{S}$, where $X \cdot A = XA$ and $A \cdot X=AX$ for all $X \in \mathcal{S}$ and $A \in \mathcal{M}$. Let $B_0$ be a fixed matrix in $M_2 (\mathbb{Q})$.  We define a derivation $\delta$ on $\mathcal{S}$, where $\delta (A)=[B_0,A]=B_0A-AB_0$ for all $A \in \mathcal{S}$. Given a bimodule homomorphism $f: \mathcal{S} \rightarrow \mathcal{M}$ over $\mathcal{S}$ by $f(A) = -A$ for all  $A \in \mathcal{S}$. We can prove that $D: \mathcal{S} \rightarrow \mathcal{M}$ is a Jordan $(\delta,f)$-derivation on module $\mathcal{M}$ over $\mathcal{S}$ by $D(A) = AB_0$ for all $A \in \mathcal{S}$.
\end{example}

The following lemma shows that every $(\delta,f)$-derivation $D$ on Jordan bimodule $\mathcal{M}$ is a Jordan $(\delta,f)$-derivation on $\mathcal{M}$.

\begin{lemma}\label{lemma3.3.1}
    Let $R$ be a commutative ring with identity, $\mathcal{S}$ an associative algebra over $R$ with identity and $\mathcal{M}$ a $2$-torsion free bimodule over $\mathcal{S}$. Let $\delta$ be a derivation on $\mathcal{S}$ and $f : \mathcal{S} \rightarrow \mathcal{M}$ a bimodule homomorphism over $\mathcal{S}$. If $D : \mathcal{S} \rightarrow \mathcal{M}$ is a $(\delta,f)$-derivation  on Jordan bimodule $\mathcal{M}$ where $D(x \bullet y) = D(x) \bullet' y + f(x) \bullet' \delta (y)$ for all $x,y \in \mathcal{S} $, then $D$ is a Jordan $(\delta,f)$-derivation on $\mathcal{M}$, i.e. $D(x^2)=D(x)x+f(x)\delta(x)$ for all $x\in \mathcal{S}$.
\end{lemma}
\begin{proof}
    Suppose $D$ is a $(\delta,f)$-derivation on Jordan bimodule $\mathcal{M}$. For any $x,y \in \mathcal{S}$ we have $D(x \bullet y) = D(x) \bullet' y + f(x) \bullet' \delta(y)$. Now we show that $D(x^2) = D(x)x+f(x)\delta (x)$ by considering the following:
    \begin{align*}
        D(x \bullet y) &= D(x) \bullet' y + f(x) \bullet' \delta (y)\\
        D(xy+yx) &= D(x)y + yD(x) + f(x)\delta (y) + \delta (y) f(x).
    \end{align*}
    Let $y=x$, obtain
    \begin{align*}
         &D(x^2+ x^2) = D(x)x + xD(x) + f(x)\delta (x) + \delta (x) f(x)\\
         &2D(x^2) = 2D(x)x + 2f(x)\delta (x)\\
         &2(D(x^2) - D(x)x - f(x) \delta(x)) = 0\\
         &D(x^2)= D(x)x+f(x)\delta (x).
    \end{align*}
    Thus, if $D$ is a $(\delta,f)$-derivation on Jordan bimodule $\mathcal{M}$, then $D$ is a Jordan $(\delta,f)$-derivation on $\mathcal{M}$.
\end{proof}

By Definition \ref{def3.3.(1)}, if $D$ is a $(\delta,f)$-derivation, then $D$ is also a Jordan $(\delta,f)$-derivation. But the converse is not always true. Hence, a sufficient condition is provided to show that if every $D$ is a Jordan $(\delta,f)$-derivation on $\mathcal{M}$ then $D$ is a $(\delta,f)$-derivation on $\mathcal{M}$. Before the result, we show some important properties which are useful.  
\begin{lemma}
     Let $\mathcal{M}$ be a $2$-torsion free bimodule over $\mathcal{S}$, $\delta$ a derivation on $\mathcal{S}$ and $f : \mathcal{S} \rightarrow \mathcal{M}$ a bimodule homomorphism over $\mathcal{S}$. If $D:\mathcal{S} \rightarrow \mathcal{M}$ is a Jordan $(\delta,f)$-derivation on $\mathcal{M}$, then $D(b)a+f(b)\delta(a)=bD(a)+\delta(b)f(a)$ for all $a,b \in \mathcal{S}$.
\end{lemma}
\begin{proof}
      Suppose $D: \mathcal{S} \rightarrow \mathcal{M}$ is a Jordan $(\delta,f)$-derivation on $\mathcal{M}$. Take any $x,y \in \mathcal{S}$, we obtain
    \begin{align}
        &D(x^2+x \bullet y+ y^2) \nonumber \\
        &= D((x+y)^2) \nonumber \\
        &=D(x+y)(x+y) + f(x+y)\delta (x+y) \nonumber \\
        &= D(x)x+D(x)y+D(y)x+D(y)y+f(x)\delta (x)+f(x)\delta(y)+f(y)\delta(x)+f(y)\delta(y). \label{eq3.3.1}
    \end{align}
    On the other hand
    \begin{align}
        &D(x^2+x \bullet y + y^2) \nonumber \\
        &= D(x^2) + D(x \bullet y) + D(x^2) \nonumber \\
        &= D(x)x+f(x)\delta (x) + D(x \bullet y) + D(y)y+f(y)\delta(y). \label{eq3.3.2}
    \end{align}
    Based on Equations (\ref{eq3.3.1}) and (\ref{eq3.3.2}) it follows that
    \begin{align*}
        D(x \bullet y) = D(x)y + D(y)x+f(x)\delta(y) + f(y)\delta (x).
    \end{align*}
    According to Lemma \ref{lemma3.3.1}, if $D(x \bullet y)=D(x) \bullet' y+f(x) \bullet' \delta(y)$ for all $x,y \in \mathcal{S}$ then $D$ is a Jordan $(\delta,f)$-derivation on  $\mathcal{M}$. So
    \begin{align*}
        D(x)y + D(y)x+f(x)\delta(y) + f(y)\delta (x) = D(x)y+yD(x)+f(x)\delta(y)+\delta(y)f(x).
    \end{align*}
    Therefore
    \begin{align*}
        D(y)x+f(y)\delta(x) = yD(x)+\delta(y)f(x).
    \end{align*}
\end{proof}

\begin{lemma}\label{Lemma3.3.3}
    Let $\mathcal{M}$ be a $2$-torsion free bimodule over $\mathcal{S}$, $\delta$ a derivation on $\mathcal{S}$ and $f : \mathcal{S} \rightarrow \mathcal{M}$ a bimodule homomorphism over $\mathcal{S}$. If $D: \mathcal{S} \rightarrow \mathcal{M}$ is a Jordan $(\delta,f)$-derivation on $\mathcal{M}$, then $D(xyx)= D(x)yx + f(x)\delta (y)x + f(x)y\delta(x)$ for all $x,y \in \mathcal{S}$.
\end{lemma}
\begin{proof}
    Let $x,y \in \mathcal{S}$. By the assumption, $D$ is a Jordan $(\delta,f)$-derivation. We have
    \begin{align}
        &D(x \bullet ( x \bullet y)) \nonumber \\
        &= D(x)\bullet' (x \bullet y) + f(x) \bullet' \delta (x \bullet y)\nonumber \\
        &= D(x) \bullet' (xy+yx) + f(x) \bullet' (\delta (x) \bullet y+ x \bullet \delta(y)) \nonumber \\
        &= D(x) (xy+yx) + (xy+yx) D(x) + f(x) \bullet' (\delta(x)y+y\delta(x)+x\delta(y)+\delta(y)x)\nonumber \\
        &= D(x)xy+D(x)yx +xyD(x) + yxD(x) + f(x)\delta(x)y+f(x)y\delta(x)+f(x)x\delta(y)\nonumber\\
        &\quad +f(x)\delta(y)x+\delta(x)yf(x)+y\delta(x)f(x)+x\delta(y)f(x)+\delta(y)xf(x).\label{eq3.3.3}
    \end{align}
    On the other hand
    \begin{align}
        &D(x \bullet ( x \bullet y)) \nonumber \\
        &= D(x^2 \bullet y) +2D(xyx) \nonumber \\
        &= D(x^2) \bullet' y + f(x^2) \bullet' \delta(y) + 2D(xyx)\nonumber \\
        &= D(x^2) y + yD(x^2) + f(x^2) \delta (y) + \delta (y) f(x^2) + 2D(xyx) \nonumber \\
        &= D(x)xy+f(x)\delta(x)y + yD(x)x+yf(x)\delta(x) + f(x)x\delta(y) + \delta(y)f(x)x+2D(xyx)\label{eq3.3.4}
    \end{align}
    Based on Equations (\ref{eq3.3.3}) and (\ref{eq3.3.4}), it follows that 
    \begin{align*}
        &2D(xyx) = 2D(x)yx+2f(x)\delta (y) x+ 2 f(x)y\delta (x)\\
        &2(D(xyx)-D(x)yx-f(x)\delta (y)x- f(x)y\delta (x)) = 0\\
        &D(xyx) = D(x)yx + f(x)\delta (y)x + f(x)y\delta (x).
    \end{align*}
\end{proof}

\begin{lemma}\label{lemma3.3.4}
    Let $\mathcal{M}$ be a $2$-torsion free bimodule over $\mathcal{S}$, $\delta$ a derivation on $\mathcal{S}$ and $f : \mathcal{S} \rightarrow \mathcal{M}$ a bimodule homomorphism over $\mathcal{S}$. If $D:\mathcal{S} \rightarrow \mathcal{M}$ is a Jordan $(\delta,f)$-derivation on $\mathcal{M}$, then
    \begin{align*}
        D(xyz+zyx)= D(x)yz + D(z)yx+f(x)\delta (y) z+ f(z)\delta (y)x + f(x) y \delta (z) + f(z)y \delta (x)
    \end{align*}
    for all $x,y,z \in \mathcal{S}$.
\end{lemma}
\begin{proof}
    Let $x,y,z \in \mathcal{S}$. Based on Lemma \ref{Lemma3.3.3}, we obtain
    \begin{align}
        &D((x+z)y(x+z)) \nonumber \\
        &= D(xyx) + D(xyz+zyx) + D(zyz) \nonumber \\
        &= D(x)yx + f(x)\delta (y)x + f(x) y \delta (x) + D(xyz+zyx)+D(z)yz + f(z)\delta (y) z + f(z) y \delta (z). \label{eq3.3.5}
    \end{align}
    On the other hand
    \begin{align}
        &D((x+z)y(x+z)) \nonumber \\
        &= D(x+z)y(x+z)+f(x+z)\delta(y)(x+z)+f(x+z)y\delta(x+z)\nonumber \\
        &= D(x)yx + D(x)yz + D(z)yx + D(z) yz + f(x) \delta (y) x + f(x) \delta (y) z + f(z) \delta (y)x \nonumber \\
        & \quad + f(z)\delta (y) z + (f(x) + f(z))(y \delta (x) + y \delta (z)) \nonumber \\
        &= D(x)yx + D(x) yz + D(z) yx + D(z)yz + f(x) \delta (y) x + f(x)\delta (y) z + f(z)\delta (y) x \nonumber\\
        &\quad +f(z)\delta(y)z+ f(x) y\delta(x)+ f(x) y \delta (z) + f(z)y \delta (x) + f(z) y \delta (z). \label{eq3.3.6}
    \end{align}
   We conclude from Equation (\ref{eq3.3.5}) and (\ref{eq3.3.6}) that
    \begin{align*}
        D(xyz+zyx)= D(x)yz + D(z)yx+f(x)\delta (y) z+ f(z)\delta (y)x + f(x) y \delta (z) + f(z)y \delta (x)
    \end{align*}
    for all $x,y,z \in \mathcal{S}$.
\end{proof}

\begin{lemma}\label{lemma3.3.5}
    Let $\mathcal{M}$ be a $2$-torsion free bimodule over $\mathcal{S}$, $\delta$ a derivation on $\mathcal{S}$ and $f : \mathcal{S} \rightarrow \mathcal{M}$ a bimodule homomorphism over $\mathcal{S}$. If $D:\mathcal{S} \rightarrow \mathcal{M}$ is a Jordan $(\delta,f)$-derivation on $\mathcal{M}$, then 
    \begin{align*}
        (D(xy)-D(x)y-f(x)\delta (y))(xy-yx)=0
    \end{align*}
    for all $x,y \in \mathcal{S}$.
\end{lemma}
\begin{proof}
    Let $x,y \in \mathcal{S}$. Consider that $\mathcal{W}=D(xy(xy)+(xy)yx)$. It follows from  Lemma \ref{lemma3.3.4} that
    \begin{align}
        \mathcal{W} &= D(xy(xy)+(xy)yx)\nonumber \\
        &= D(x)y(xy) + D(xy)yx + f(x) \delta (y) xy + f(xy)\delta (y) x + f(x)y \delta (xy) + f(xy)y\delta (x)\nonumber \\
        &= D(x) y (xy) + D(xy) yx + f(x) \delta (y) xy + f(x)y\delta (y)x + f(x) y \delta (xy) + f(x)y^2 \delta (x). \label{eq3.3.7}
    \end{align}
    In another way, we have
    \begin{align}
        \mathcal{W} &= D(xy(xy)+(xy)yx) \nonumber \\
        &= D(xy)(xy) + f(xy)\delta (xy) + D(x)y^2x + f(x)\delta (y^2) x + f(x)y^2 \delta (x) \nonumber \\
        &= D(xy)(xy) + f(x) y \delta (xy) + D(x) y^2 x + f(x) \delta (y) yx + f(x) y \delta (y)x + f(x) y^2 \delta (x). \label{eq3.3.8}
    \end{align}
    As a consequence of Equations (\ref{eq3.3.7}) and (\ref{eq3.3.8}), it follows that
    \begin{align*}
(D(xy)-(x)y-f(x)\delta(y))(xy-yx)=0.
    \end{align*}
\end{proof}

The following lemma is the result of linearizing Lemma \ref{lemma3.3.5}.
\begin{lemma}\label{lemma3.3.6}
    Let $\mathcal{M}$ be a $2$-torsion free bimodule over $\mathcal{S}$, $\delta$ a derivation on $\mathcal{S}$ and $f : \mathcal{S} \rightarrow \mathcal{M}$ a bimodule homomorphism over $\mathcal{S}$. If $D:\mathcal{S} \rightarrow \mathcal{M}$ is a Jordan $(\delta,f)$-derivation on $\mathcal{M}$, then
    \begin{align*}
         (D(xy)-D(x)y-f(x)\delta (y))(zy-yz) + (D(zy)-D(z)y-f(z)\delta (y))(xy-yx) = 0
    \end{align*}
    and
    \begin{align*}
        (D(xy)-D(x)y-f(x)\delta(y))(zx-xz)+(D(zx)-D(z)x-f(z)\delta(x))(yx-xy) =0
    \end{align*}
    for all $x,y,z \in \mathcal{S}$.
\end{lemma}

\begin{lemma}\label{lemma3.3.7}
    Let $\mathcal{M}$ be a $2$-torsion free bimodule over $\mathcal{S}$, $\delta$ a derivation on $\mathcal{S}$ and $f : \mathcal{S} \rightarrow \mathcal{M}$ a bimodule homomorphism over $\mathcal{S}$. If $D:\mathcal{S} \rightarrow \mathcal{M}$ is a Jordan $(\delta,f)$-derivation on $\mathcal{M}$ and $xy=xy$ for all $x,y \in \mathcal{S}$, then
    \begin{align*}
        D(xy)=D(x)y+f(x)\delta(y).
    \end{align*}
\end{lemma}
\begin{proof}
    Let $x,y \in \mathcal{S}$. Let $\mathcal{W}=D(xy(xy)+(xy)yx)$. It is known that $xy=yx$, thus
    \begin{align}
        S &= D(xy(xy)+(xy)yx) \nonumber \\
        &= 2D((xy)^2)\nonumber \\
        &= 2D(xy)(xy) + 2f(xy)\delta (xy) \nonumber \\
        &= 2D(xy) (xy) + f(x) y \delta (xy) + f(x)y \delta (y)x + f(x)y^2 \delta (y). \label{eq3.3.9}
    \end{align}
    Alternatively, Lemma \ref{lemma3.3.4} yields 
    \begin{align}
        S &= D(xy(xy)+(xy)yx) \nonumber \\
        &= D(x)y(xy)+D(xy)yx+f(x)\delta (y) xy+f(xy)\delta(y)x + f(x)y \delta (xy) + f(xy)y\delta (x) \nonumber \\
        &= D(x)y(xy) + D(xy)xy + f(x)\delta (y)xy + f(x)y \delta (y)x + f(x)y\delta (xy) + f(x)y^2\delta (x). \label{eq3.3.10}
    \end{align}
    By combining Equations (\ref{eq3.3.9}) and (\ref{eq3.3.10}), we obtain
    \begin{align*}
        &D(xy)(xy) = D(x) y(xy) + f(x)\delta (y) (xy)\\
        & (D(xy)-D(x)y-f(x)\delta (y))(xy) = 0\\
        & D(xy) = D(x)y + f(x)\delta (y).
    \end{align*}
\end{proof}

Now, we recall the definition of jointly prime bisubmodules from \cite{khumprapussorn} but in special case, i.e $\mathcal{A}$ is an associative algebra over $R$ with identity.
\begin{definition}
    A proper bisubmodule $\mathcal{K}$ of $\mathcal{M}$ is called jointly prime if for each left ideal $I$ of $\mathcal{S}$, right ideal $J$ of $\mathcal{S}$ and bisubmodule $\mathcal{N}$ of $\mathcal{M}$,
    \begin{align*}
        I \mathcal{N} J \subseteq \mathcal{K} \text{ implies } I \mathcal{M} J \subseteq \mathcal{K} \text{ or } \mathcal{N} \subseteq \mathcal{K}.
    \end{align*}
    A bimodule $\mathcal{M}$ is said to be jointly prime if its zero bisubmodule is a jointly prime bisubmodule of $\mathcal{M}$.
\end{definition}

\begin{lemma}\label{lemma3.3.8}
    Let $\mathcal{M}$ be a $2$-torsion free and jointly prime bimodule over $\mathcal{S}$, $\delta$ a derivation on $\mathcal{S}$ and $f : \mathcal{S} \rightarrow \mathcal{M}$ a bimodule homomorphism over $\mathcal{S}$. If $D:\mathcal{S} \rightarrow \mathcal{M}$ is a Jordan $(\delta,f)$-derivation on $\mathcal{M}$, and $xy=0$ for all $x,y \in \mathcal{S}$, then
    \begin{align*}
        D(xy) = D(x)y + f(x)\delta (y)=0.
    \end{align*}
\end{lemma}
\begin{proof}
     Let any $x,y \in \mathcal{S}$ and it is known that $xy = 0$. If $yx = 0$ then $xy=yx = 0$ and based on Lemma \ref{lemma3.3.7} implies that 
     \begin{align*}
         D(xy) & =D(x)y + f(x)\delta (y)\\
        0 &= D(x)y + f(x)\delta (y).
     \end{align*}
     If $yx \ne 0$ and according to Lemma \ref{lemma3.3.6}, we known that
     \begin{align*}
         &(D(xy)-D(x)y-f(x)\delta(y))(zx-xz) - (D(zx)-D(z)x-f(z)\delta (x)) (yx-xy)=0\\
         &y(D(xy)-D(x)y-f(x)\delta (y))(zxy-xzy) - y(D(zx)-D(z)x-f(z)\delta (x))(yxy-xyy)=0\\
         &y(D(xy)-D(x)y-f(x)\delta(y))(xzy)=0 \quad \quad \text{for all } z \in \mathcal{S}.
     \end{align*}
     Since $\mathcal{M}$ is a jointly prime bimodule, $yx \ne 0$ and we assumed that $y \ne 0$, it implies that
     \begin{align*}
         (D(xy)-D(x)y-f(x)\delta (y))x &= 0\\
          -(D(x)y+f(x)\delta (y))x &= 0.
     \end{align*}
     Let $y = yry$. If $xy = 0$, then for all $r \in \mathcal{S}, x(yry)=0$. Substituting $y = yry$ gives
     \begin{align*}
          &-(D(x)y+f(x)\delta (y))x=0\\
          &(D(x)yry+f(x)\delta (yry))x = 0\\
          & (D(x)yry + f(x)\delta (y)ry+f(x)y\delta (r)y + f(x)yr\delta (y))x = 0\\
          &(D(x)yry+f(x)\delta(y)ry+f(xy)\delta (r)y+f(xy)r \delta(y))x = 0\\
          &(D(x)y+f(x)\delta(y))ryx = 0.
     \end{align*}
     Since $\mathcal{M}$ is a jointly prime bimodule and $yx \ne 0$, it follows that $D(x)y + f(x)\delta (y) = 0$. Therefore, it is proven that for all $x,y \in \mathcal{S}$ and $xy=0$, we have
     \begin{align*}
         D(xy)=D(x)y+ f(x)\delta (y)=0.
     \end{align*}
\end{proof}

\begin{corollary}\label{cor3.3.1}
   Let $\mathcal{M}$ be a $2$-torsion free and jointly prime bimodule over $\mathcal{S}$, $\delta$ a derivation on $\mathcal{S}$ and $f : \mathcal{S} \rightarrow \mathcal{M}$ a bimodule homomorphism over $\mathcal{S}$. If $D:\mathcal{S} \rightarrow \mathcal{M}$ is a Jordan $(\delta,f)$-derivation on $\mathcal{M}$, and for all $x,y \in \mathcal{S}, xy=0$, then
    \begin{align*}
        D(yx)= D(y)x+\delta(y)f(x).
    \end{align*}
\end{corollary}
\begin{proof}
    Let any $x,y \in \mathcal{S}$. Given that $xy=0$, we will prove that $D(yx)=yD(x)+f(y)\delta(x)$. Based on Lemma \ref{lemma3.3.8}, we obtain $D(x)y+f(x)\delta(y)=0$. It is noted that
    \begin{align*}
        D(yx) &= D(yx+xy) \nonumber\\
        &=D(y \bullet x)\nonumber\\
        &=D(y) \bullet' x + f(y) \bullet' \delta (x)\nonumber\\
        &= D(y) x+xD(y) + f(y) \delta (x) + \delta(x)f(y)\nonumber\\
        &= D(y)x+f(y)\delta(x)+D(x)y+f(x)\delta(y)\\
        &= D(y)x+f(y)\delta(x).
    \end{align*}
\end{proof}

\begin{lemma}\label{lemma3.3.9}
     Let $\mathcal{M}$ be a $2$-torsion free and jointly prime bimodule over $\mathcal{S}$, $\delta$ a derivation on $\mathcal{S}$ and $f : \mathcal{S} \rightarrow \mathcal{M}$ a bimodule homomorphism over $\mathcal{S}$. If $D:\mathcal{S} \rightarrow \mathcal{M}$ is a Jordan $(\delta,f)$-derivation on $\mathcal{M}$, and for all $x,y \in \mathcal{S}, xy=0$, then for all $z \in \mathcal{S}$,
    \begin{align*}
        D((yx)z)=D(yx)z+f(yx)\delta(z).
    \end{align*}
\end{lemma}
\begin{proof}
    Let $x,y,z \in \mathcal{S}$. Let $\mathcal{H}=D(zxy+yxz)$. Assume that $xy=0$, we obtain
    \begin{align*}
        \mathcal{H} &= D(zxy + yxz)\\
        &= D(z)xy + D(y) xz + f(z)\delta (x)y + f(y)\delta(x)z + f(z)x\delta(y) + f(y) x \delta (z)\\
        &= (D(y)x + f(y)\delta(x))z + f(z)(\delta(x)y+x\delta(y))+f(yx)\delta(z)\\
        &= D(yx)z+f(z)\delta(xy) + f(yx)\delta(z)\\
        &= D(yx)z + f(yx)\delta(z).
    \end{align*}
    Therefore, it is proven that if $x,y \in \mathcal{S}$ and $xy=0$ then for all $z \in \mathcal{S}$
    \begin{align*}
        D((yx)z)=D(yx)z + f(yx)\delta(z).
    \end{align*}
\end{proof}

The following corollary is an immediate consequence of Lemma \ref{lemma3.3.9}.

\begin{corollary}\label{cor3.3.2}
     Let $\mathcal{M}$ be a $2$-torsion free and jointly prime bimodule over $\mathcal{S}$, $\delta$ a derivation on $\mathcal{S}$, $f : \mathcal{S} \rightarrow \mathcal{M}$ a bimodule homomorphism over $\mathcal{S}$. Let $D:\mathcal{S} \rightarrow \mathcal{M}$ be a Jordan $(\delta,f)$-derivation on $\mathcal{M}$ and $\mathcal{P}=\{ x \in \mathcal{S} : D(xa) = D(x)a+f(x)\delta(a), \text{ for all } a \in \mathcal{S}\}$. If $xy=0$, then based on Lemma \ref{lemma3.3.9}, we obtain $D((yx)z)=D((yx)z)+D(yx)z+f(yx)\delta(z)$ for all $z \in \mathcal{S}$ such that $yx \in \mathcal{P}$.
\end{corollary}

\begin{theorem}\label{teo3.3.1}
     Let $\mathcal{M}$ be a $2$-torsion free and jointly prime bimodule over $\mathcal{S}$, $\delta$ a derivation on $\mathcal{S}$ and $f : \mathcal{S} \rightarrow \mathcal{M}$ a bimodule homomorphism over $\mathcal{S}$. Let $D:\mathcal{S} \rightarrow \mathcal{M}$ be a Jordan $(\delta,f)$-derivation on $\mathcal{M}$. If $x^y=D(xy)-D(x)y-f(x)\delta(y)$ for all $x,y \in \mathcal{S}$, then the following hold:
    \begin{enumerate}
        \item $x^{y+z} = x^y + x^z$,
        \item $x^y = -y^x$,
    \end{enumerate}
    for all $x,y,z \in \mathcal{S}$.
\end{theorem}
\begin{proof}
    \begin{enumerate}
        \item Given that $x^y = D(xy)-D(x)y-f(x)\delta(y)$ for all $x,y \in \mathcal{S}$. It is noted that for all $x,y,z \in \mathcal{S}$,
        \begin{align*}
            x^y &= D(xy)-D(x)y-f(x)\delta(y),\\
            x^z &= D(xz)-D(x)z-f(x)\delta(z).
        \end{align*}
        Thus, we obtain
        \begin{align*}
            x^y + x^z &= D(xy)-D(x)y-f(x)\delta(y)+D(xz)-D(x)z-f(x)\delta(z)\\
            &= D(xy)+D(xz)-D(x)(y+z)-f(x)(\delta(y)+\delta(z))\\
            &= D(xy+xz)-D(x)(y+z)-f(x)(\delta(y+z))\\
            &= D(x(y+z))-D(x)(y+z)-f(x)\delta(y+z)\\
            &= x^{y+z}.
        \end{align*}
        \item For all $x,y \in \mathcal{S}$, we have $x^y = D(xy)-D(x)y-f(x)\delta(y)$ and $y^x=D(yx)-D(y)x-f(y)\delta(x)$. Thus, we obtain
        \begin{align*}
            x^y + y^x &= D(xy)-D(x)y-f(x)\delta(y) + D(yx)-D(y)x-f(y)\delta(x)\\
            x^y + y^x &= D(xy+yx) - D(y)x-f(x)\delta(y) -D(y)x-f(y)\delta(x)\\
            x^y+y^x &= D(x)y + D(y)x + f(x)\delta(y)+f(y)\delta(x) -D(x)y-f(x)\delta(y)\\
            &\quad -D(y)x-f(y)\delta(x)\\
            x^y+y^x &= 0\\
            x^y &= -y^x.
        \end{align*}
    \end{enumerate}
\end{proof}

\begin{lemma}\label{lemma3.3.10}
    For $x \in \mathcal{S}$, the set $T(x)$ is defined as $T(x) = \{ m \in \mathcal{M} \mid m(xa-ax)=0 \text{ for all } a \in \mathcal{S} \}$. Then, $T(x)$ is a subbimodule of $\mathcal{M}$.
\end{lemma}
\begin{proof}
   We show that $T(x)$ is a subbimodule of $\mathcal{M}$. First, it will be shown that $T(x) \neq  \emptyset$. Note that $0 \in \mathcal{M}$, so for all $s \in \mathcal{S},0(xa-ax)=0xa-0ax=0$. Hence, we obtain $T(x)  \neq \emptyset.$  Let $v_1,v_2 \in T(x)$ and $a \in \mathcal{S}$. Based on the definition of $T(x)$, we have $v_1(xa-ax)=0$ and $v_2(xa-ax)=0$. Consider that
    \begin{align*}
       (v_1-v_2)(xa-ax) &=  v_1 (xa-ax)-v_2(xa-ax) \\
        & = 0+0 =0.
    \end{align*}
    Thus, we obtain $v_1-v_2 \in T(x)$. Next, let $v_1 \in T(x)$ and $a \in \mathcal{S}$, we have
    \begin{align*}
        (v_1)(xar-arx) &= 0 \quad \quad \text{ for all } r \in \mathcal{S}\\
        v_1 \{ (xa-ax)r+a(xr-rx)\} & = 0\\
        v_1(xa-ax)r +v_1a(xr-rx) &= 0\\
        v_1a(xr-rx) &= 0
    \end{align*}
    and
    \begin{align*}
        av_1(xr-rx) &= a(v_1(xr-rx))=0.
    \end{align*}
    Thus, we obtain $v_1a, av_1\in T(x)$. Therefore, We conclude that $T(x)$ is a subbimodule of $\mathcal{M}$.
\end{proof}

\begin{lemma}\label{lemma3.3.11}
     Let $\mathcal{M}$ be a $2$-torsion free and jointly prime bimodule over $\mathcal{S}$. If $x \in \mathcal{S}$ and $x \notin Z(\mathcal{S})$, then $T(x) =0$.
\end{lemma}
\begin{proof}
    Given that $x \in \mathcal{S}$ and $x \notin Z(\mathcal{S})$, it will be shown that $T(x)=0$. Since $x \notin Z(\mathcal{S})$, for some $p \in \mathcal{S}$, we have $xp-px \ne 0$. Suppose $T(x) \ne 0$, then there exists $m \in T(x)$ with $m \ne 0$ such that $m(xp-px)=0$ for all $p \in \mathcal{S}$. Since $\mathcal{M}$ is a bimodule over $\mathcal{S}$, for any $c,d \in \mathcal{S}$ it follows that $(cmd)(xp-px)=c(m(xp-px))d=0$. This means, the bisubmodule $\mathcal{S}m\mathcal{S} \subseteq \mathcal{M}$ satisfies $\mathcal{S}m\mathcal{S} (xp-px)=0$. Consider $I=\mathcal{S}$ be a left ideal  of $\mathcal{S}$, $J=\mathcal{S}$ a right ideal of $\mathcal{S}$ and $N=\mathcal{S}m\mathcal{S}$ a bisubmodule of $\mathcal{M}$. Because $\mathcal{M}$ is a jointly prime bimodule over $\mathcal{S}$ and $xp-px\ne 0$, it must be that $\mathcal{S}m\mathcal{S} = 0$, which implies $m =0$. This contradiction. Therefore, it is proven that if $x \in \mathcal{S}$ and $x \notin Z(\mathcal{S})$, then $T(x)=0$.
\end{proof}

\begin{lemma}\label{lemma3.3.12}
    Let $\mathcal{M}$ be a $2$-torsion free and jointly prime bimodule over $\mathcal{S}$, $\delta$ a derivation on $\mathcal{S}$ and $f : \mathcal{S} \rightarrow \mathcal{M}$ a bimodule homomorphism over $\mathcal{S}$. Let $D:\mathcal{S} \rightarrow \mathcal{M}$ be a Jordan $(\delta,f)$-derivation on $\mathcal{M}$ and  $\mathcal{P}=\{ x \in \mathcal{S} : D(xa) = D(x)a+f(x)\delta(a), \text{ for all } a \in \mathcal{S}\}$. If $u \in \mathcal{P}$ and $u \notin Z(\mathcal{S})$, and if $vu=uv$, then $v \in \mathcal{P}$.
\end{lemma}
\begin{proof}
    Given that $u \in \mathcal{P}, u \notin Z(\mathcal{S})$ and $vu=uv$, it will be shown that $v \in \mathcal{P}$. Let $x,y,z \in \mathcal{S}$, based on Lemma \ref{lemma3.3.6}, we know that
    \begin{align*}
        & (D(xy)-D(x)y-f(x)\delta(y))(yz-zy)+(D(zy)-D(z)y-f(z)\delta(y))(xy-yx) = 0\\
        & x^y (yz-zy)+z^y (xy-yx)=0.
    \end{align*}
    If $u=z \in \mathcal{P}$ then $u^y = D(uy)-D(u)y -f(u)\delta(y) =0$ for all $y \in \mathcal{S}$. Therefore
    \begin{align*}
        x^y (yu-uy)+u^y(xy-yx)&=0\\
        x^y (uy-yu) &= 0 \\
        -x^y (uy-yu)&=0 \quad \text{ for all } x,y \in \mathcal{S}.
    \end{align*}
    Next, let any $y',z'\in \mathcal{S}$. Let $y=y'+z$. Substituting $y=y'+z'$ gives
    \begin{align*}
        &-x^{(y'+z')}(u(y'+z')-(y'+z')u) = 0\\
        &(x^{y'} + x^{z'})(uy'+uz'-y'u-z'u)=0\\
        &x^{y'}(uy'-y'u)+x^{y'}(uz'-z'u)+x^{z'}(uy'-y'u)+x^{z'}(uz'-z'u)=0\\
        &x^{y'}(uz'-z'u)+x^{z'}(uy'-y'u)=0 \quad \quad \text{ for all } x,y',z' \in \mathcal{S}.
    \end{align*}
    Since $vu=uv$ and let $y'=v$, we obtain
    \begin{align*}
        x^v (uz'-z'u)+x^{z'}(uv-vu) &=0\\
        x^v (uz'-z'u) + x^{z'}(uv-vu) &= 0\\
        x^v (uz'-z'u) &= 0.
    \end{align*}
    Therefore, $xv \in T(u)$. Since $u \notin Z(\mathcal{S})$, by Lemma \ref{lemma3.3.11}, we obtain $T(u) = 0$ and $xv=0$ for all $x \in \mathcal{S}$. Since $xv=-vx$, it follows that $vx=0$ for all $x \in \mathcal{S}$. Therefore, we conclude that $v \in \mathcal{P}$.
\end{proof}

We present the definition of prime algebra.
\begin{definition}
    \cite{erickson}
    Let $R$ be a commutative ring with identity and $\mathcal{S}$ an associative algebra over $R$ with identity. An algebra $\mathcal{S}$ over $R$ is said to be prime if for any two $R$-ideals $U$ and $V$ of $\mathcal{S}, UV=0$ implies $U=0$ or $V=0$.
\end{definition}

\begin{lemma}\label{lemma3.3.13}
    Let $\mathcal{S}$ be an associative prime algebra over $\mathcal{R}$ with identity, $\mathcal{M}$ a $2$-torsion free and jointly prime bimodule over $\mathcal{S}$, $\delta$ a derivation on $\mathcal{S}$ and $f : \mathcal{S} \rightarrow \mathcal{M}$ a bimodule homomorphism over $\mathcal{S}$. Let $D:\mathcal{S} \rightarrow \mathcal{M}$ be a Jordan $(\delta,f)$-derivation on $\mathcal{M}$ and $\mathcal{P}=\{ x \in \mathcal{S} : D(xa) = D(x)a+f(x)\delta(a), \text{ for all } a \in \mathcal{S}\}$. If $s \in \mathcal{S}$ and $s^2 = 0$ then $s \in \mathcal{P}$.
\end{lemma}
\begin{proof}
    Let any $s \in \mathcal{S}$. Given that $s^2 = 0$, it will be shown that $s \in \mathcal{P}$. Let any $r \in \mathcal{S}$, so that $s \cdot sr = 0$. Let $H = D(qssr+srsq)$. It is noted that 
    \begin{align*}
        H &= D(qssr + srsq)\\
        &= D(q)ssr + D(sr)sq + f(q)\delta(s)sr + f(sr)\delta(s)q + f(q) s\delta(sr) + f(sr) s\delta(q)\\
        &= (D(sr)s+f(sr)\delta(s))q + f(q)(\delta(s)sr+s\delta(sr))+f(sr)s\delta(q)\\
        &= D(srs)q + f(q)\delta(ssr)+f(sr)s\delta(q)\\
        &= D(srs)q + f(srs)\delta(q) \quad \quad \text{ for all } q \in \mathcal{S}.
    \end{align*}
    Therefore, based on Corollary \ref{cor3.3.2}, we obtain $srs \in \mathcal{S}$. If we assume $s \ne 0$, then $srs \ne 0$ for some $r \in \mathcal{S}$. Next, observe that $s(srs) =s^2rs=0$ and $(srs)s=srs^2=0$, so we have $s(srs)=(srs)s=0$. Since $srs \ne 0$ but $(srs)(srs)=srs^2rs=0$, and since $\mathcal{S}$ is a prime algebra and $srs \notin Z(\mathcal{S})$, it follows that by Lemma \ref{lemma3.3.12} with $u=srs \in \mathcal{P}$ and $v=s$, we obtain $s \in \mathcal{P}$.
\end{proof}

\begin{lemma}\label{lemma3.3.14}
    Let $\mathcal{S}$ be an associative prime algebra over $\mathcal{R}$ with identity, $\mathcal{M}$ a $2$-torsion free and jointly prime bimodule over $\mathcal{S}$, $\delta$ a derivation on $\mathcal{S}$ and $f : \mathcal{S} \rightarrow \mathcal{M}$ a bimodule homomorphism over $\mathcal{S}$. Let $D:\mathcal{S} \rightarrow \mathcal{M}$ be a Jordan $(\delta,f)$-derivation on $\mathcal{M}$ and $\mathcal{P}=\{ x \in \mathcal{S} : D(xa) = D(x)a+f(x)\delta(a), \text{ for all } a \in \mathcal{S}\}$. If $z,w \in \mathcal{P}$ then $y^x (zw-wz)=0$ for all $x,y \in \mathcal{S}$.
\end{lemma}
\begin{proof}
    Let any $x',y,z,w \in \mathcal{S}$. Based on Lemma \ref{lemma3.3.6}, we obtain $y^{x'}(zx'-x'z)+z^{x'}(yx'-x'y)=0$. Since $z \in \mathcal{P}$, we know that $z^{x'}=D(zx')-D(z)x'-f(z)\delta(x')=0$ for all $x' \in \mathcal{S}$, thus we obtain $y^{x'}(zx'-x'z)=0$. Let $x \in \mathcal{S}$ and $x'=x+w$. we obtain
    \begin{align*}
        &y^{x+w}(z(x+w)-(x+w)z) =0\\
        & y^x (zx-xz)+y^x(zw-wz)+y^w(zx-xz)+y^w(zw-wz)=0\\
        &y^x (zw-wz) + y^w (zx-xz) = 0.
    \end{align*}
    Since $w \in \mathcal{P}$, we know that $w^y = 0$ such that $y^w = 0$. Therefore,
    \begin{align*}
        y^x (zw-wz) + y^w (zx-xz) &= 0\\
        y^x (zw-wz) &= 0.
    \end{align*}
\end{proof}

The following theorem shows that every Jordan $(\delta,f)$-derivation is a $(\delta,f)$-derivation.
\begin{theorem}\label{teo3.3.2}
    Let $R$ be a commutative ring with identity, $\mathcal{S}$ an associative prime algebra over $R$ with identity and $\mathcal{M}$ a $2$-torsion free and jointly prime bimodule over $\mathcal{S}$. Let $\delta$ be a derivation on $\mathcal{S}$ and $f: \mathcal{S} \rightarrow \mathcal{M}$ a bimodule homomorphism over $\mathcal{S}$. If $D:\mathcal{S} \rightarrow \mathcal{M}$ be a Jordan $(\delta,f)$-derivation on $\mathcal{M}$ then $D$ is a $(\delta,f)$-derivation on $\mathcal{M}$.
\end{theorem}
\begin{proof}
    Let any $s \in \mathcal{S}$ such that $s^2 =0$. Based on Lemma \ref{lemma3.3.12}, we obtain $s \in \mathcal{P}$. Let $q \in \mathcal{S}$ such that $q^2 =0$, then we obtain $q \in \mathcal{P}$. Next, based on Lemma \ref{lemma3.3.13},
    \begin{align*}
        y^x (qs-sq) &= 0.
    \end{align*}
    Let $q=y$, so $y^x(ys-sy)=0$. Therefore $y^x \in T(y)$. If $y \notin Z(\mathcal{S})$, then $T(y)=0$. So $y^x=D(yx)-D(y)x-f(y)\delta(x)=0$. If $y \in Z(\mathcal{S})$, then $yx=xy$ for all $x \in \mathcal{S}$. Based Lemma \ref{lemma3.3.7}, we have $y^x=D(yx)-D(y)x-f(y)\delta(x)=0$. It can be concluded that $y^x =0$. It is noted that
    \begin{align*}
        &y^x = -x^y\\
        &D(xy)-D(x)y-f(x)\delta(y)= 0\\
        &D(xy)=D(x)y+f(x)\delta(y).
    \end{align*}
    Thus, it is proven that every $D$ as a Jordan $(\delta,f)$-derivation on $\mathcal{M}$ is a $(\delta,f)$-derivation on $\mathcal{M}$.
\end{proof}

\section*{Acknowledgements}

The authors would like to express their sincere gratitude to the anonymous reviewers for their insightful comments and valuable suggestions, which have significantly improved the quality and clarity of this manuscript. Their constructive feedback helped strengthen the arguments and address important aspects of the research.

\section*{ORCID}

\noindent Indah Emilia Wijayanti - \url{https://orcid.org/0000-0003-0390-8682}


\begin{thebibliography}{00}
%1
\bibitem{shakir}
S. Ali, N.N. Rafiquee and V. Varshney, Certain types of derivations in rings: a survey, {\it Journal of the Indonesian Mathematical Society}, {\bf 30} (2) (2024) 256-306.

\bibitem{bland}
P. E. Bland, $f$-derivations on rings and modules, {\it Comment. Math. Univ. Carolin}, {\bf 47} (3) (2006) 379-390.

\bibitem{creedon}
T. Creedon, Derivations and prime ideals, {\it Proc. Royal Irish Acad}. {\bf 98} (1998) 223-225.
%2
\bibitem{cusack}
J. M. Cusack, Jordan derivations on ring, {\it Proceedings of the American Mathematical Society}, {\bf 53} (2) (1975) 321-324.

\bibitem{erickson}
T. Erickson, W. Martindale, \& J. Osborn, Prime nonassociative algebras, {\it Pacific Journal of Mathematics}, {\bf 60} (1) (1975) 49-63.

\bibitem{esslamzadeh}
G. H. Esslamzadeh and H. Ghahramani, Existence, automatic continuity and invariant sub-modules of generalized derivations on modules, {\it Aequat. Math}. {\bf 84} (2012) 185–200.



%5
\bibitem{fitriani}
Fitriani, I. E. Wijayanti, A. Faisol, and S. Ali, On $f$-derivations on Polynomial Modules, {\it Journal of Algebra and Its Applications}. (2024) 2550155.

%6
\bibitem{ghahramani}
H. Ghahramani, M. N. Ghosseiri and T. Rezaei, Posner's first theorem for prime modules, {\it Khayyam Journal of Mathematics}. {\bf 10} (1) (2024) 165-175.


%9
\bibitem{heirstein}
I. N. Herstein, Jordan derivations of prime rings, {\it Proc. Amer. Math. Soc}. {\bf 8} (1957) 1104-1110.

\bibitem{li}
Y. Li, \& Benkovi$\check{c}$, Jordan generalized derivations on triangular algebras, {\it Linear and Multilinear Algebra}. {\bf 59}(8) (2011), 841-849.

\bibitem{indah}
I. E. Wijayanti, \& R. Wisbauer, On Coprime Modules and Comodules, {\it Communications in Algebra}. {\bf 37} (4) (2009) 1308-1333.
%10
\bibitem{jacobson}
N. Jacobson, Structure of rings, {\it American Mathematical Society}. (1956).

\bibitem{khumprapussorn}
T. Khumprapussorn, S. Pianskool, \& M. Hall, Modules and their Fully and Jointly Prime Submodules, {\it In International Mathematical Forum}. {\bf 7} (33) (2012) 1631-1643.



%12
\bibitem{posner}
E. C. Posner, Derivations in Prime Rings, {\it Proceedings of the American Mathematical Society}. {\bf 8} (6) (1957) 1093-1110.


\end{thebibliography}
\end{document}